\documentclass[12pt]{article}
\usepackage{amsrefs}
\usepackage{amsmath,amssymb,amsthm,enumerate}
\usepackage{tikz,graphicx,mathdesign}
\usepackage[all,cmtip]{xy}
\usepackage[applemac]{inputenc}

\def \1{{\bf 1}}
\def \al{\alpha}

\def \ds{\displaystyle}
\def \e{\emph}
\def \eps{\varepsilon}

\def \ol{\overline}

\def \supp{\operatorname{supp}}

\def \({\left(}
\def \){\right)}

\newtheorem{theorem}{Theorem}[section]
\newtheorem{conjecture}[theorem]{Conjecture}
\newtheorem{lemma}[theorem]{Lemma}

\newtheorem{proposition}[theorem]{Proposition}
\theoremstyle{definition}

\newtheorem{definition}[theorem]{Definition}

\begin{document}

\pagestyle{myheadings} \markright{HAAR SYSTEMS FOR GROUPOIDS}

\title{On Haar systems for groupoids}
\author{Anton Deitmar}
\date{}
\maketitle

{\bf Abstract:} It is shown that a locally compact groupoid with open range map does not always admit a Haar system.
It then is shown how to construct a Haar system if the stability groupoid and the quotient by the stability groupoid both admit one.


$$ $$

\tableofcontents

\newpage
\section*{Introduction}

Topological groupoids occur naturally in encoding hidden symmetries like in fundamental groupoids or holonomy groupoids of foliations, see \cite{Paterson}, for instance.
In order to construct convolution algebras on groupoids \cites{Renault,RenaultCK}, one needs continuous families of invariant measures, so called \e{Haar systems} \cite{Seda1}, see also Section \ref{SecHaar}.
These do not always exist.
One known criterion is that a Haar system can only exist if the range map is open  (Corollary to Lemma 2 in \cite{Seda}, see also \cite{Williams}).

A second criterion, which has been neglected in the literature, is the possibility of \e{failing support}, i.e., it is possible that, although the range map is open, the support condition of a Haar system cannot be satisfied, see Proposition \ref{Prop4.1}.
We conjecture, however, that there should always be a Haar system for a locally compact groupoid with open range map, if the groupoid is second countable. 

We show how to construct Haar systems if the stability groupoid and its quotient both admit one.

I thank Dana Williams for some very helpful comments.

\section{Locally compact groupoids}\label{SecHaar}

\begin{definition}
By a \e{bundle of groups} we understand a continuous map $\pi:G\to X$ between locally compact Hausdorff spaces together with a group structure on each fibre 
$G_x=\pi^{-1}(x)$, $x\in X$ such that the following maps are continuous:
\begin{align*}
\eps:X\to G&&&\text{identity,}\\
m:G^{(2)}\to G&&&\text{multiplication,}\\
\iota:G\to G&&&\text{inverse},
\end{align*}
where $G^{(2)}$ is the set of all $(x,y)\in G\times G$ with $\pi(x)=\pi(y)$.
\end{definition}

Note that this implies that $\eps$ is a homeomorphism onto the image, so $X$ carries the subspace topology but also $X$ carries the quotient topology induced by the surjective map $\pi$.
In all, the topology on $X$ is determined by the one on $G$.

\begin{definition}
Each fibre $G_x$ being a locally compact group, carries a Haar measure which is unique up to scaling.
A \e{coherent system} of Haar measures is a family $(\mu_x)_{x\in X}$ such that $\mu_x$ is a Haar measure on $G_x$ such that for each $\phi\in C_c(G)$ the map
$$
x\mapsto \int_{G_x}\phi\,d\mu_x
$$
is continuous.
\end{definition}

\begin{proposition}\label{thm2.4}
Let $\pi:G\to X$ be a bundle of groups over a paracompact space $X$.
There exists a coherent system of Haar measures $\mu_x$ on $G$ if and only if the map $\pi$ is open.
\end{proposition}

\begin{proof}
This is Lemma 1.3 in \cite{RenaultIdeal}.
\end{proof}

\begin{definition}
Let $X$ be a set.
By a \e{groupoid} over $X$ we mean a category with object class $X$ (so it is a small category) in which each arrow is an isomorphism.
We write $G$ for the set of arrows and we use the following notation
\begin{align*}
r,s:G\to X &&&\text{range and source maps,}\\
\eps:X\to G&&&\text{identity,}\\
G^{(2)}\subset G\times G&&&\text{set of composable pairs,}\\
m:G^{(2)}\to G&&&\text{composition,}\\
\iota:G\to G&&&\text{inverse}.
\end{align*}
\end{definition}

\begin{definition}
A \e{topological groupoid} is a groupoid $G$ over $X$ together with topologies on $G$ and $X$ such that the structure maps $r,s,\eps,m,\iota$ are continuous.
Here $G\times G$ carries the product topology and $G^{(2)}\subset G\times G$ the subset topology.
Note that if $X$ is Hausdorff, then $G^{(2)}=\{ (\al,\beta)\in G\times G: r(\beta)=s(\al)\}$ is a closed subset of $G\times G$.
\end{definition}

A \e{locally compact groupoid} is a topological groupoid such that $G$ and $X$ are locally compact Hausdorff spaces.

From now on $G$ is assumed to be a locally compact groupoid.
We use the  notation
\begin{align*}
G_x&=\{ g\in G: s(g)=x\},\\
G^y&=\{ g\in G: r(g)=y\},\\
G_x^y&=G_x\cap G^y.
\end{align*}
As $X$ is Hausdorff, all three sets are closed in $G$.

Note that a bundle of groups is a special case of a groupoid $G$ with $G_x^y=\emptyset$ if $x\ne y$.

\begin{definition}
For a groupoid $G$ the \e{stability groupoid} is defined to be the subset
$$
G'=\big\{ g\in G: r(g)=s(g)\big\}.
$$ 
If $G$ is a topological groupoid, then $G'$ is a closed subgroupoid. 
\end{definition}

\begin{definition}
On a groupoid $G$ we install an equivalence relation 
$$
g\sim h\quad\Leftrightarrow\quad r(g)=r(h)\text{ and }s(g)=s(h).
$$
we write $[g]$ for the equivalence class, i.e., $[g]=G_{s(g)}^{r(g)}$.
\end{definition}

Now assume that $(\mu_x^x)_{x\in X}$ is a coherent family of measures on the  bundle of groups  $G'=\{ g\in G:r(g)=s(g)\}$.
We then get invariant measures $\mu_{[g]}$ on the classes $[g]$ by setting
$$
\int_{[g]}\phi(x)\,d\mu_{[g]}(x)=\int_{G_{s(g)}^{s(g)}}\phi(gx)\,d\mu_{s(g)}^{s(g)}(x).
$$
The invariance of the $\mu_x^x$ yields the well-definedness of the $\mu_{[g]}$.
The uniqueness of a Haar measure implies that $\mu_{[g]}$ is, up to scaling, the unique Radon measure on $[g]$ being right-invariant under $G_{s(g)}^{s(g)}$ or left-invariant under $G_{r(g)}^{r(g)}$.

In the sequel, we shall identify a Radon measure with its positive linear functional, so we write $\mu_{[g]}(\phi)$ for the above integral.

\begin{definition}
We shall need the notion of a topological right-action of a topological groupoid $H$ on a topological space $Z$.
This is given by the following data: first there is a continuous surjection $\rho:Z\to X$, where $X$ is the base set of $H$. We define
$$
Z*H=\big\{ (z,h): \rho(z)=r(h)\big\}.
$$
This is a closed subset of $Z\times H$ and we consider it equipped with the corresponding topology.
Next the action is given by a map
\begin{align*}
Z*H&\to Z,\\
(z,h)&\mapsto zh,
\end{align*}
such that $\rho(zh)=s(h)$ and $z\cdot 1=z$ as well as $z(hh')=(zh)h'$ holds for all $(z,h),(z,hh')\in Z*H$.

Note that the action defines an equivalence relation on $Z$ given by $z\sim zh$ for $h\in H$. We naturally equip $Z/H$ with the quotient topology.
\end{definition}

\begin{lemma}\label{lem2.9}
Assume the locally compact groupoid $H$ acts on a locally compact space $Z$ and that $H$ has open range map.
Then the projection $Z\to Z/H$ is open.
\end{lemma}

\begin{proof}
This is Lemma 2.1 in \cite{MW}. However, in that paper the assertion was given under a stronger definition of $H$-actions then the one we use, as it was assumed that the map $\rho:Z\to X$ also be open.
Lemma 2.1 and its proof in \cite{MW}, however, are valid under our weaker assumptions.
For the convenience of the reader we shall show this by reproducing the proof here:
Let $V\subset Z$ be open, in order to show that its image in $Z/H$ is open, it suffices to show that the union of orbits
$
VH=\big\{ vh:v\in V, (v,h)\in Z*H\big\}
$
is open in $Z$.
So it suffices to show that any net $z_i\to vh$ with $v\in V$ and $h\in H$ eventually is in $VH$.
But $\rho(z_i)$ converges to $\rho(vh)=s(h)$.
As the range map of $H$ is open, so is the source map $s$,  hence the set $s(H)$ is open and we can find a net $h_i$ in $H$ on the same index set, such that $\rho(z_i)=s(h_i)$ for all $i\ge i_0$ for some index $i_0$.
Further, the same applies to open neighborhoods of $h$, so we can choose the net so that $h_i\to h$.
Then $z_ih_i^{-1}$ converges to $v$ and thus  is eventually  in $V$ and $z_i=z_ih_i^{-1}h_i$ is eventually in $VH$.
\end{proof}

\begin{definition}
An action of a groupoid $H$ on a space $Z$  is called \e{free} if $zh=z$ implies that $h=1_{s(g)}$ and it is called \e{proper}, if the map $Z*H\to Z\times Z$, $(z,h)\mapsto (zh,z)$ is proper.

For any groupoid $G$ the action of $G'$ on $G$ is easily seen to be free and proper.
\end{definition}

\begin{lemma}\label{prop3.4}
Let $G$ be a locally compact groupoid over a paracompact space $X$ and let $(\mu_x^x)_{x\in X}$ be a coherent system of Haar measures on the groups $G_x^x$, $x\in X$.
Then for every $\phi\in C_c(G)$ the function
$$
\ol\phi:g\mapsto \mu_{[g]}(\phi)
$$
is continuous.
\end{lemma}

\begin{proof}
Since the $G'$ action is free and proper, this is immediate from Lemma 2.9 of \cite{MRW}.
\end{proof}

\section{Haar systems}

\begin{definition}
A \e{Haar system} on the locally compact groupoid $G$ is a family $(\mu^x)_{x\in X}$ of Radon measures on $G$ with
\begin{enumerate}[\rm (a)]
\item $\supp(\mu^x)=G^x$,
\item $\ds\int_G\phi(\al g)\,d\mu^{y}(g)=
\int_G\phi(g)\,d\mu^{x}$ for every $\phi\in C_c(G)$ and every $\al\in G_y^x$,
\item $\ds x\mapsto \int_G\phi(g)\,d\mu^x(g)$ is continuous on $X$ for every $\phi\in C_c(G)$.
\end{enumerate}
\end{definition}

If a locally compact groupoid $G$ admits a Haar system, then the range map, and so the source map, too, is open, see Corollary to Lemma 2 in \cite{Seda}, see also \cite{Williams}.

The question for the converse assertion, asked in \cite{Williams}, is answered in the negative by the following proposition.

\begin{proposition}\label{Prop4.1}
There exists a locally compact, even compact, groupoid $G$, whose range map is open, but no  Haar system exists on $G$.
\end{proposition}

\begin{proof}
There are locally compact, even compact, Hausdorff spaces which cannot be the support of any Radon measure.
Here are two examples:
\begin{itemize}
\item Let $X$ be the unit ball of a Hilbert space of uncountable dimension and equip $X$ with the weak topology.
By the Banach-Alaoglu-Theorem, $X$ is a compact Hausdorff space.
By Corollary 7.14.59 of volume 2 of \cite{Boga}, the set $X$ cannot be the support of any Radon measure
\item (Williams) Let $Y$ be an uncountable set with the discrete topology and let $X=Y\cup\{\infty\}$ be its one-point compactification.
Then $X$ cannot be the support of any Radon measure. To see this, let $m$ be a Radon measure on $X$, then $m(X)<\infty$, as $X$ is compact. Further, $m(Y)=\sum_{y\in Y}m(\{y\})$, as $m$ is regular and the only compact subsets of $Y$ are the finite sets.
As $m(Y)<\infty$, the set $M$ of all $y\in Y$ with  $m(\{y\})>0$ is countable, therefore $M\ne Y$ and $m$ is supported in $M\cup\{\infty\}$.
\end{itemize}

Let now $X$ be any  locally compact Hausdorff space which is not the support of a Radon measure.
Let $G=X\times X$ with the product topology and make $g$ a groupoid by setting $(x,y)(y,z)=(x,z)$ and $r(x,y)=x$ as well as $s(x,y)=y$.
Then the source map is a homeomorphism between $G^x$ and $X$, so $G^x$ cannot be the support of any Radon measure, hence no Haar system exists.
\end{proof}

\begin{conjecture}
Every second countable, locally compact groupoid with open range map admits a Haar system. 
\end{conjecture}

\begin{definition}
Let $G$ be a groupoid over $X$. 
We write $E(G)\subset X\times X$ for the image of the map $g\mapsto(s(g),r(g))$.
Then $E(G)$ is an equivalence relation on $X$.

We say that a groupoid $G$ is a \e{principal groupoid} if $G_x^x=\{1_x\}$ for every $x\in X$.
This means that the groupoid is completely described by its equivalence relation.
Note, though, that for topological groupoids the topology on $G$ generally differs from the one on $E(G)$ as a subset of $X\times X$.
\end{definition}

\begin{lemma}
Let $G$ be a groupoid over a set $X$.
Define an equivalence relation on $G$ by
$$
g\sim h\quad\Leftrightarrow\quad r(g)=r(h)\text{ and }s(g)=s(h).
$$
Then the set $\ol G=G/\sim$ becomes a groupoid, indeed a principal groupoid, by setting $[g][h]=[gh]$ whenever $g$ and $h$ are composable.
\end{lemma}

\begin{proof}
This is easily checked.
\end{proof}

\begin{theorem}\label{thm5.5}
Let $G$ be a locally compact groupoid over a paracompact space $X$.
Suppose that the stability groupoid $G'$ has open range map.
\begin{enumerate}[\rm (a)]
\item The groupoid $\ol G$, when equipped with the quotient topology, is a locally compact groupoid. The quotient map $G\to \ol G$ is open.
\item If the range map of $G$ is open, then so is the range map of $\ol G$.
\item If $\ol G$ admits a Haar system, then $G$ admits a Haar system.
\end{enumerate}
\end{theorem}

\begin{proof}
(a) 
By Proposition \ref{thm2.4}, the groupoid $G'$ admits a coherent system of Haar measures $(\mu_x^x)_{x\in X}$.
Let $g_0\in G$ and let $\phi\in C_c^+(G)$ such that $\phi(g_0)>0$.
Let
$$
\ol\phi:g\mapsto\int_{G_{s(g)}^{r(g)}}\phi(gh)\,d\mu_{s(g)}^{r(g)}(h).
$$
By Lemma \ref{prop3.4} the map $\ol\phi$ is continuous.
It factors over $\ol G$, hence defines a continuous map of compact support on $\ol G$.
The set $U=\{ x\in\ol G: \ol\phi(x)>0\}$ is an open neighborhood of $[g_0]$, so $\supp(\ol\phi)$ is a compact neighborhood of $[g_0]$.
Therefore $\ol G$ is locally compact.

If $[g]\ne [h]$, then we can find $\phi,\psi\in C_c^+(G)$ such that $\ol\phi$ and $\ol\psi$ have disjoint supports and $\phi(g),\psi(h)>0$.
Considering the continuous function $\ol\phi-\ol\psi$ on $\ol G$, one sees that $[h]$ and $[g]$ have disjoint neighborhoods, so $\ol G$ is a Hausdorff space.
Together we infer that $\ol G$ is a locally compact groupoid.

The quotient map $p:G\to\ol G$ is open by Lemma \ref{lem2.9}.

(b) As the range map of $G$ is open and factors over the range map of $\ol G$, the range map of $\ol G$ is open as well.

(c) If $(m^x)$ is a Haar system for $\ol G$, then
$$
\phi\mapsto \int_{\ol G}\ol\phi(g)\,dm^x(g)
$$
defines a Haar system on $G$.
\end{proof}

{\small Mathematisches Institut,
Auf der Morgenstelle 10,
72076 T\"ubingen,
Germany\\
\tt deitmar@uni-tuebingen.de}

\today


\begin{bibdiv} \begin{biblist}


\bib{Boga}{book}{
   author={Bogachev, V. I.},
   title={Measure theory. Vol. I, II},
   publisher={Springer-Verlag, Berlin},
   date={2007},
   pages={Vol. I: xviii+500 pp., Vol. II: xiv+575},
   isbn={978-3-540-34513-8},
   isbn={3-540-34513-2},
   review={\MR{2267655}},
   doi={10.1007/978-3-540-34514-5},
}

\bib{HA2}{book}{
   author={Deitmar, Anton},
   author={Echterhoff, Siegfried},
   title={Principles of harmonic analysis},
   series={Universitext},
   publisher={Springer},
   place={New York},
   date={2009},
   pages={xvi+333},
   isbn={978-0-387-85468-7},
}

\bib{Morita}{article}{
   author={Morita, K.},
   title={Paracompactness and product spaces},
   journal={Fund. Math.},
   volume={50},
   date={1961/1962},
   pages={223--236},
   issn={0016-2736},
}

\bib{Paterson}{book}{
   author={Paterson, Alan L. T.},
   title={Groupoids, inverse semigroups, and their operator algebras},
   series={Progress in Mathematics},
   volume={170},
   publisher={Birkh\"auser Boston, Inc., Boston, MA},
   date={1999},
   pages={xvi+274},
   isbn={0-8176-4051-7},
   doi={10.1007/978-1-4612-1774-9},
}


\bib{Renault}{book}{
   author={Renault, Jean},
   title={A groupoid approach to $C^{\ast} $-algebras},
   series={Lecture Notes in Mathematics},
   volume={793},
   publisher={Springer, Berlin},
   date={1980},
   pages={ii+160},
   isbn={3-540-09977-8},
}

\bib{RenaultIdeal}{article}{
   author={Renault, Jean},
   title={The ideal structure of groupoid crossed product $C^\ast$-algebras},
   note={With an appendix by Georges Skandalis},
   journal={J. Operator Theory},
   volume={25},
   date={1991},
   number={1},
   pages={3--36},
   issn={0379-4024},
}

\bib{RenaultCK}{article}{
   author={Kumjian, Alex},
   author={Pask, David},
   author={Raeburn, Iain},
   author={Renault, Jean},
   title={Graphs, groupoids, and Cuntz-Krieger algebras},
   journal={J. Funct. Anal.},
   volume={144},
   date={1997},
   number={2},
   pages={505--541},
   issn={0022-1236},
   doi={10.1006/jfan.1996.3001},
}

\bib{MRW}{article}{
   author={Muhly, Paul S.},
   author={Renault, Jean N.},
   author={Williams, Dana P.},
   title={Equivalence and isomorphism for groupoid $C^\ast$-algebras},
   journal={J. Operator Theory},
   volume={17},
   date={1987},
   number={1},
   pages={3--22},
   issn={0379-4024},
}

\bib{MW}{article}{
   author={Muhly, Paul S.},
   author={Williams, Dana P.},
   title={Groupoid cohomology and the Dixmier-Douady class},
   journal={Proc. London Math. Soc. (3)},
   volume={71},
   date={1995},
   number={1},
   pages={109--134},
   issn={0024-6115},
   review={\MR{1327935}},
   doi={10.1112/plms/s3-71.1.109},
}

\bib{Ramsay}{article}{
   author={Ramsay, Arlan},
   title={The Mackey-Glimm dichotomy for foliations and other Polish
   groupoids},
   journal={J. Funct. Anal.},
   volume={94},
   date={1990},
   number={2},
   pages={358--374},
   issn={0022-1236},
   doi={10.1016/0022-1236(90)90018-G},
}


\bib{Rudin}{article}{
   author={Rudin, Mary Ellen},
   title={A new proof that metric spaces are paracompact},
   journal={Proc. Amer. Math. Soc.},
   volume={20},
   date={1969},
   pages={603},
   issn={0002-9939},
}

\bib{Seda1}{article}{
   author={Seda, A. K.},
   title={Haar measures for groupoids},
   journal={Proc. Roy. Irish Acad. Sect. A},
   volume={76},
   date={1976},
   number={5},
   pages={25--36},
}

\bib{Seda}{article}{
   author={Seda, Anthony Karel},
   title={On the continuity of Haar measure on topological groupoids},
   journal={Proc. Amer. Math. Soc.},
   volume={96},
   date={1986},
   number={1},
   pages={115--120},
   issn={0002-9939},
   doi={10.2307/2045664},
}

\bib{Vaughan}{article}{
   author={Vaughan, H. E.},
   title={On locally compact metrisable spaces},
   journal={Bull. Amer. Math. Soc.},
   volume={43},
   date={1937},
   number={8},
   pages={532--535},
   issn={0002-9904},
   doi={10.1090/S0002-9904-1937-06593-1},
}


\bib{Williams}{article}{
   author={Williams, Dana P.},
   title={Haar systems on equivalent groupoids},
   journal={Proc. Amer. Math. Soc. Ser. B},
   volume={3},
   date={2016},
   pages={1--8},
   issn={2330-1511},
   review={\MR{3478528}},
}

\end{biblist} \end{bibdiv}
\end{document}